\documentclass[a4paper, 12pt, reqno]{amsart}

\numberwithin{equation}{section}

\usepackage{amssymb}
\usepackage{mathrsfs}
\usepackage{dsfont}
\usepackage{stmaryrd, MnSymbol}
\usepackage{enumerate, xspace}
\usepackage{changebar}

\newtheorem{theorem}{Theorem}
\newtheorem{lemma}{Lemma}[section]
\newtheorem{corollary}{Corollary}[section]
\newtheorem{proposition}{Proposition}[section]

\newtheorem{remark}{Remark}[section]
\newtheorem{example}{Example}[section]

\numberwithin{equation}{section}

\renewcommand{\a}{\alpha}
\renewcommand{\b}{\beta}

\newcommand{\La}{\Lambda}

\def\R{{\mathbb{R}}}
\def\N{{\mathbb{N}}}
\def\Z{{\mathbb{Z}}}

\def\F{{\mathcal{F}}}

\begin{document}

\begin{small}
*The first version of the present paper (\cite{S16}) was almost finished 3 years ago, but has not been submitted. In the mean time, Gilbert and Gaspard \cite{GG16}
show how the variational characterization can be put to use to obtain the correction to static (or instantaneous) part of the diffusion coefficient and carried out further molecular dynamics simulations, which on one side confirm our picture and on the other hand also show that the correction is very small.
The present version is a substantial generalization of \cite{S16}.
\end{small}

\bigskip
\bigskip

\title[Energy Diffusion]
{The gradient condition and the contribution of the dynamical part of Green-Kubo formula to the diffusion coefficient}
\author{Makiko Sasada}
\address{Makiko Sasada\\
Graduate School of Mathematical Sciences, The University of Tokyo\\
3-8-1, Komaba, Meguro-ku, Tokyo, 153-8914, Japan}
\email{{\tt sasada@ms.u-tokyo.ac.jp}}

\begin{abstract}
In the diffusive hydrodynamic limit for a symmetric interacting particle system (such as the exclusion process, the zero range process, the stochastic Ginzburg-Landau model, the energy exchange model), a possibly non-linear diffusion equation is derived as the hydrodynamic equation. The bulk diffusion coefficient of the limiting equation is given by Green-Kubo formula and it can be characterized by a variational formula. In the case the system satisfies the gradient condition, the variational problem is explicitly solved and the diffusion coefficient is given from the Green-Kubo formula through a static average only. In other words, the contribution of the dynamical part of Green-Kubo formula is $0$. In this paper, we consider the converse, namely if the contribution of the dynamical part of Green-Kubo formula is $0$, does it imply the system satisfies the gradient condition or not. We show that if the equilibrium measure $\mu$ is product and $L^2$ space of its single site marginal is separable, then the converse also holds.

As an application of the result, we consider a class of stochastic models for energy transport studied by Gaspard and Gilbert in \cite{GG08,GG09}, where the exact problem is discussed for this specific model. 

\end{abstract}


\thanks{ This paper has been partially supported by by the Grant-in-Aid for Young Scientists (B) 25800068.}

\maketitle

\section{Introduction}\label{aba:sec1}
In the study of the hydrodynamic limit for a large scale of interacting particle systems, the system is said to satisfy the gradient condition, if the current of the conserved quantity is given by a linear sum of the difference of a local function and its space-shift. If the system satisfies the gradient condition, the diffusion coefficient of the hydrodynamic equation has an explicit expression and the proof of the scaling limit becomes much simpler than the general case (cf. \cite{KL}). The underlying structure for this simplification is that the gradient condition implies that the contribution of the dynamical part of Green-Kubo formula is $0$. Then, it might be natural to ask whether the converse statement holds or not. Namely, if the contribution of the dynamical part of Green-Kubo formula is $0$, does it imply the system satisfies the gradient condition? Though the question sounds very natural, we could not find any explicit answer in the literature. In this paper, we give the answer under the assumption that the equilibrium measure is a product measure. Our motivation originally comes from the series of papers by Gaspard and Gilbert \cite{GG08,GG09} where the relation of the gradient condition and the contribution of the dynamic part of GK-formula was discussed. In the last section of the paper, we show an application of our result to this model.

The proof of our main result relies on very fundamental observations for non-dynamical problems. More precisely, the key theorem (Theorem \ref{thm:main} below) concerns only about the properties of the equilibrium measure. 

In the next section, we give our general setting and state the main result. In Section \ref{sec:proof}, we give a proof of Theorem \ref{thm:main}. For simplicity we first discuss about the one-dimensional case and then generalize it to the higher dimensional case. In the last section, we explain an application to the model studied by Gaspard and Gilbert in \cite{GG08,GG09}.


\section{Setting and main result}

We consider a general interacting particle system with stochastic dynamics, whose state space is given by a product space $\Omega=X^{\Z^d}$ where $X$, the single component space, is a measurable space. We suppose that $\Omega$ is the product measurable space equipped with a translation invariant probability measure $\mu$ and denote the expectation with respect to $\mu$ by $\langle \cdot \rangle$ and the inner product of $L^2(\mu)$ by $\langle \cdot, \cdot \rangle$. We denote by $(\eta_x)_{x \in \Z^d}$ the element of $\Omega$. A measurable function $f: \Omega \to \R$ is called local if it depends only on a finite number of coordinates, and for a local function $f$, we define $s_f:=\min\{n \ge 0; f \textit{does not depend on } (\eta_x)_{|x| \ge n+1} \}$ where $|x|=\max\{|x_1|,|x_2|,\dots, |x_d|\}$ for $x \in \Z^d$.
Shift operators $\tau_z$ are defined for each $z \in \Z^d$ as $(\tau_z \eta)_x=\eta_{x-z}$ and $(\tau_z f) (\eta)= f(\tau_z \eta)$.
 
Let $\mathcal{D}:=\{f \in L^2(\mu); f : \textit{local} \} $. If an operator $T:\mathcal{D} \to \mathcal{D}$ satisfies that there exists $r \ge 0$ such that $s_{Tf} \le \max\{s_f,r\}$, then we call it a local operator.

We consider a set of local operators $(L_{x,y})_{x,y \in \Z^d}$ satisfies $L_{x,y} \mathbf{1}=0$ and the following conditions with convention $L_{x,x}\equiv 0$:
\begin{itemize}
\item Translation invariance : $L_{x,y}=\tau_{x} L_{0,y-x} \tau_{-x}$
\item Finite range : There exists $R >0$ such that $L_{0,z} \equiv 0$ for $|z| >R$
\item Symmetry : $L_{0,z}=L_{z,0}$ for any $z \in \Z^d$
\item Reversibility : $\langle L_{0,x}f, g \rangle=\langle f, L_{0,x}g \rangle$ for $f,g \in  \mathcal{D}$
\item Non-positivity: $\mathcal{D}_{0,x}(f):=\langle -L_{0,x}f, f \rangle \ge 0$ for $f \in  \mathcal{D}$
\end{itemize}

We suppose that $L=\sum_{x,y \in \Z^d}L_{x,y}$ defines the Markov process $\{\eta_x(t)\}_{x \in \Z^d}$ whose (formal) generator is $L$ with initial distribution $\mu$. We do not attempt here at a justification of this setting in full generality but rather refer to the examples for full rigor. By the reversibility, $\mu$ is the stationary measure for the process.

Our interest is in the case where the conservation quantity exists. Actually, we also suppose that there exists a measurable function $\xi :X \to \R$ such that $\xi(\eta_0) \in L^2(\mu)$ and 
$L_{x,y}(\xi_x+\xi_y)=0 $ and $ L_{x,y} \xi_z=0$ for $z \neq x,y$ where $\xi_x:=\xi(\eta_x)$.

\begin{example}
The exclusion process with a proper jump rate $c$ is in our setting with $X=\{0,1\}$, $L_{x,y}f=\frac{1}{2}c(x,y,\eta)(f(\eta^{x,y})-f(\eta))$, $\mu$: a product Bernoulli measure and $\xi(\eta)=\eta$ where $c(x,y,\eta)=c(y,x,\eta)$ and $\eta^{x,y}$ is the configuration obtained from $\eta$ by exchanging the occupation variables at $x$ and $y$.
\end{example}

\begin{example}
The generalized exclusion process is in our setting with $X=\{0,1,\dots,\kappa\}$, 
\[
L_{x,y}f=\mathbf{1}_{ \{|x-y|=1\}}\frac{1}{2}\big( \mathbf{1}_{\{\eta_x \ge 1, \eta_y \le \kappa-1 \}}(f(\eta^{x \to y})-f(\eta)) +\mathbf{1}_{\{\eta_y \ge 1, \eta_x \le \kappa-1 \}}(f(\eta^{y \to x})-f(\eta)) \big),
\]
$\mu$: a translation invariant Gibbs measure and $\xi(\eta)=\eta$ where $\eta^{x \to y }$ is the configuration obtained from $\eta$ by letting a particle jump from $x$ to $y$ (cf. \cite{KL}).
\end{example}

\begin{example}
The zero-range process with a proper jump rate $g$ is in our setting with $X=\{0,1,2,\dots,\}=\Z_{\ge 0}$, $L_{x,y}f=\mathbf{1}_{\{|x-y|=1\}}\frac{1}{2}(g(\eta_x)(f(\eta^{x \to y})-f(\eta))+g(\eta_y)(f(\eta^{y \to x})-f(\eta)))$, $\mu$: a product Gibbs measure and $\xi(\eta)=\eta$ (cf. \cite{KL}).
\end{example}

\begin{example}
The stochastic Ginzburg-Landau process with proper functions $a$ and $V$ is in our setting with $X=\R$, $L_{x,y}=(\partial_{\eta_x}-\partial_{\eta_y})(a(\eta_x,\eta_y)(\partial_{\eta_x}-\partial_{\eta_y}))+a(\eta_x,\eta_y)(V'(\eta_x)-V'(\eta_y))(\partial_{\eta_x}-\partial_{\eta_y})$, $\mu$: a product Gibbs measure given by the potential $V$ and $\xi(\eta)=\eta$ (cf. \cite{Lu,V}).
\end{example}

\begin{example}
The stochastic energy exchange model is also in our setting as shown in Section \ref{sec:energy}.
\end{example}

Let $\langle \xi_0 \rangle =\rho$, $S(x,t):=\mathbb{E}[\xi_x(t) \xi_0(0)] -\rho^2$ and $\chi:=\sum_{x \in \Z^d}S(x,t)$, which we suppose finite. As studied in \cite{Sp} (Section 2.2 of Part II) for exclusion processes, the bulk diffusion coefficient matrix $D=(D_{\a\b})$ for the conserved quantity $\xi$ is defined as 
\begin{displaymath}
D_{\a\b}:=\lim_{t \to \infty} \frac{1}{t}\frac{1}{2\chi}\sum_{x \in \Z^d}x_{\a}x_{\b}S(x,t),
\end{displaymath}
$\a,\b=1,2,\dots,d$.
 
Under a general condition, we can show the following Green-Kubo formula (cf. \cite{Sp} (Section 2.2 of Part II)):
\begin{equation}\label{GKformula}
D_{\a\b}=\frac{1}{2\chi} \big( 2\sum_{x}x_{\a}x_{\b}\mathcal{D}_{0,x}(\xi_0) - 2\int_0^{\infty} \sum_x \mathbb{E}[ j_{\a} e^{Lt} \tau_x j_{\b}] dt \big)
\end{equation}
where $j_{\a}=\sum_{x \in\Z^d} x_{\a} j_{0,x}$ and $j_{0,x}= 2L_{0,x}\xi_0=(L_{0,x}+L_{x,0})\xi_0$, which is a current between $0$ and $x$. Note that
\[
\xi_x(t)-\int_0^t \sum_{z} j_{x,x+z}(s) ds
\]
is a Martingale and $2\mathcal{D}_{0,x}(\xi_0)=-\langle \xi_0, j_{0,x} \rangle$. We call the term $\frac{1}{2\chi} \sum_{x}x_{\a}x_{\b}\mathcal{D}_{0,x}(\xi_0) $ as the static part of Green-Kubo formula and 
$-\frac{1}{\chi} \int_0^{\infty} \sum_x  \mathbb{E}[j_{\a} e^{tL} \tau_x  j_{\b}] dt$ as the dynamical part of Green-Kubo formula. We define the matrix $D^s$ as
\[
D^s_{\a\b}=\frac{1}{\chi}\sum_{x}x_{\a}x_{\b}\mathcal{D}_{0,x}(\eta_0).
\]

We introduce the Hilbert space $\mathcal{H}$ of functions on $\Omega$ as the completion of $\mathcal{D}$ equipped with the (degenerate) scalar product
\[
\langle f | g \rangle  := \sum_{x \in \Z^d} (\langle \tau_xf g \rangle -  \langle f \rangle \langle g \rangle).
\]
Here, we suppose that the measure $\mu$ satisfies an enough spatial mixing condition to make $\langle f | g \rangle$ be well-defined for any $f, g \in \mathcal{D}$. Actually, in our main theorem, we only consider product measures.
We also suppose that $e^{tL}$ induces the self-adjoint semigroup $T_t$ on $\mathcal{H}$ and denote its generator by $\tilde{L}$. 

The following variational formula also holds under a general condition (cf. \cite{Sp} (Section 2.2 of Part II), \cite{KL}):
\begin{equation}\label{variational formula}
\sum_{\a,\b=1}^d{\ell_{\a}}{\ell_{\b}}D_{\a\b}=\sum_{\a,\b=1}^d{\ell_{\a}}{\ell_{\b}}D^s_{\a\b}+ \inf_{f \in \mathcal{D}} \{ - 2  \langle \sum_{\a=1}^d \ell_{\a} j_{\a} | f \rangle -\langle f | \tilde{L} f \rangle \}
\end{equation}
for all $\ell=(\ell_{\a}) \in \R^d$.

So far, we did not prove anything and just introduce the settings. From now on, under the assumption that relations (\ref{GKformula}) and (\ref{variational formula}) hold, we state our main result. For this, we introduce the gradient space 
\[\mathfrak{G}:=\{\sum_{\a=1}^d (\tau^{\a} g_{\a} -g_{\a}) ;  g_{\a} \in \mathcal{D},\a=1,2,\dots,d \}
\] where $\tau^{\a}=\tau_{e_{\a}}$ and $e_{\a}$ is the unit vector to the $\a$-th direction. The stochastic system defined by $L$ is said to satisfy the gradient condition, if $j_{\a} \in \mathfrak{G}$ for all $\a=1,2,\dots,d$.

Our main result is that our stochastic system satisfies the gradient condition if and only if $D=D^s$ under the condition that $\mu$ is product and the $L^2$ space of its single site marginal is separable. To show this, we first give two simple lemmas.

\begin{lemma}
If the stochastic system defined by $L$ satisfies the gradient condition, then $D=D^s$. Namely, the variational formula (\ref{variational formula}) attains its minimum with $f=0$.
\end{lemma}
\begin{proof}
If $j_{\a} \in \mathfrak{G}$ for all $\a=1,2,\dots,d$, then $\langle \sum_{\a=1}^d \ell_{\a} j_{\a} | f \rangle=0$. Since $-\langle f | \tilde{L} f \rangle \ge 0$ for any $f \in \mathcal{D}$ by the positivity condition, $\sum_{\a,\b=1}^d{\ell_{\a}}{\ell_{\b}}D_{\a\b}=\sum_{\a,\b=1}^d{\ell_{\a}}{\ell_{\b}}D^s_{\a\b}$ for all $\ell=(\ell_{\a}) \in \R^d$.
\end{proof}

\begin{lemma}
If $D=D^s$, then $\langle j_{\a} | j_{\a} \rangle=0$ for all $\a=1,2,\dots,d$.
\end{lemma}
\begin{proof}
If $D=D^s$, we obtain from (\ref{variational formula}) with $\ell_{\a}=1,\ell_{\b}=0$ for $\b \neq \a$, that
\begin{equation}
\inf_{f \in \mathcal{D}} \{ - 2  \langle   j_{\a} | f \rangle -\langle f | \tilde{L} f \rangle \}=0.
\end{equation}
Since $j_{\a}=2\sum_{x}x_{\a} L_{0,x}\xi_0 \in \mathcal{D}$, we can take $c j_{\a}$ as $f$ in the above variational formula for any $c \in \R$ and obtain
\begin{equation*}
\inf_{c \in \R} \{ - 2c  \langle   j_{\a} | j_{\a} \rangle -c^2 \langle j_{\a} | \tilde{L} j_{\a} \rangle \} \ge 0
\end{equation*}
which implies $\langle   j_{\a} | j_{\a} \rangle=0$.

\end{proof}

\begin{remark}
If the interaction of our system is nearest-neighbour, namely, $R=1$, then we have $j_{\a}=j_{0,e_{\a}}-j_{0,-e_{\a}}=j_{0,e_{\a}}+\tau_{-e_{\a}}j_{0,e_{\a}}$ since $j_{0,-e_{\a}}=2L_{0,-e_{\a}}\xi_0=-2L_{0,-e_{\a}}\xi_{-e_{\a}}=-2L_{-e_{\a},0}\xi_{-e_{\a}}=-\tau_{-e_{\a}}j_{0,e_{\a}}$. For this case, $\langle   j_{\a} | j_{\a} \rangle=0$ is equivalent to $\langle   j_{0,e_\a} | j_{0,e_\a} \rangle=0$.
\end{remark}

Next theorem is the most essential result and we give its proof in the next section.

\begin{theorem}\label{thm:main}
Assume that $\mu$ is product with a single site marginal $\nu$, namely $\mu=\nu^{\Z^d}$, and $L^2(\nu)$ is separable. Then, if $f \in \mathcal{D}$ satisfies $\langle f \rangle=0$ and $\langle f | f \rangle=0$, then $f \in \mathfrak{G}$. Equivalently, the intersection of the kernel of $\langle \cdot | \cdot \rangle$ and $\mathcal{D}$ is the direct sum of the space of constant functions and $\mathfrak{G}$. 
\end{theorem}

Combining this theorem with the above lemmas, we obtain our main result as a straightforward corollary.
\begin{corollary}
Assume that $\mu$ is product with a single site marginal $\nu$, namely $\mu=\nu^{\Z^d}$, and $L^2(\nu)$ is separable. Then, the stochastic system defined by $L$ satisfies the gradient condition, if and only if $D=D^s$. Moreover, if $R=1$, then $D=D^s$ if and only if $j_{0,e_\a} \in \mathfrak{G}$ for all $\a=1,2,\dots,d$.
\end{corollary}

\begin{remark}
Separability condition for $L^2(\nu)$ is quite mild. In particular, if $\nu$ is a probability measure on the measurable space $(E, \mathbb{B}(E))$ where $E$ is a Borel set of Euclidean space and $\mathbb{B}(E)$ is the Borel set on $E$, then $L^2(\nu)$ is separable. Also, if $X$ is a countable set, then $L^2(\nu)$ is separable.
\end{remark}

\begin{remark}
Theorem \ref{thm:main} is nothing to do with the dynamics, but only concerns the probability measure $\mu$.
\end{remark}
\section{Proof of Theorem \ref{thm:main}} \label{sec:proof} 

In the first subsection, we give the proof for the case $d=1$. In the second subsection, we generalize it to the case $d\ge 2$.

\subsection{The one dimensional setting}

 
We consider the case $d=1$. Let $(X, \mathcal{F},\nu)$ be a probability space where $L^2(\nu)$ is separable and $\Omega:=X^{\Z}$ be the infinite product probability space equipped with the probability measure $\mu:=\nu^{\Z}$. Let $\mathcal{D}_0:=\{f \in L^2(\mu); f : \textit{local}, \langle f\rangle =0 \} $. For $f \in \mathcal{D}_0$, we define a semi-norm $\| \cdot \|$ as
\begin{displaymath}
\| f \|^2:=\lim_{k \to \infty}\frac{1}{2k+1} \langle (\sum_{x=-k}^k\tau_xf)^2\rangle =\sum_{x \in \Z} \langle f \tau_x f\rangle = \langle f | f\rangle .
\end{displaymath}

Theorem \ref{thm:main} concerns the relation between the gradient space $\mathfrak{G}:=\{\tau g -g ;  g \in \mathcal{D}\}=\{\tau g -g ;  g \in \mathcal{D}_0\}$ and the kernel of the semi-norm $C_0:=\{f \in \mathcal{D}_0; \|f \|=0 \}$.

It is easy to see that $\mathfrak{G} \subset C_0$. Theorem \ref{thm:main} claims that $\mathfrak{G} \supset C_0$ hence $\mathfrak{G} = C_0$.

\smallskip

To prove this, we first start with a simple lemma. 
Let 
$\ell^2_c:=\{\mathbf{a}=(a_x)_{x \in \Z} \in \R^{\Z}; |\{x \in \Z; a_x \neq 0\}| < \infty \}$. Here $|A|$ represents the number of elements for a set $A$. For $\mathbf{a} \in \ell^c_2$ satisfying $\mathbf{a} \nequiv \mathbf{0}$, define $M_{\mathbf{a}}:=\max\{ x \in \Z; a_x \neq 0\}$ and $m_{\mathbf{a}}:=\min\{ x \in \Z; a_x \neq 0\}$. As a convention, take $M_{\mathbf{0}}=m_{\mathbf{0}}
=0$. We also define $\Lambda_{\mathbf{a}}:=\{ x \in \Z; m_{\mathbf{a}} \le x \le M_{\mathbf{a}}\}$ and $s_{\mathbf{a}}:= |\Lambda_{\mathbf{a}}|$.
 
\begin{lemma}\label{lem:grad}
Let $f \in \mathcal{D}_0$ and assume that there exists $(a_x)_{x \in \Z} \in \ell^2_c$ and $h \in \mathcal{D}_0$ satisfying $f=\sum_{x} a_x \tau_x h$ and $\sum_x a_x=0$. Then, there exists a unique function $g \in \mathcal{D}_0$ such that $f=\tau g -g $, hence $f \in \mathfrak{G}$. Moreover, if $\{ \tau_x h\}_{x \in \Z}$ are orthogonal in $L^2(\mu)$, then $\langle g^2\rangle  \le s_{\mathbf{a}}^2 \sum_{x \in \Z}a_x^2 \langle h^2\rangle $.
\end{lemma}
\begin{proof}
Uniqueness: If $f=\tau g_1 -g_1=\tau g_2 -g_2 $ and $g_1,g_2 \in \mathcal{D}_0$, then $\tau(g_1-g_2) =g_1-g_2$ and $g_1-g_2 \in \mathcal{D}_0$. In particular, $g_1-g_2$ is local and shift invariant, so it must be a constant. Also, $\langle g_1-g_2\rangle =0$, hence $g_1 \equiv g_2$.

Existence: Since $\sum_{x \in \Z} a_x= \sum_{x=m_{\mathbf{a}}}^{M_{\mathbf{a}}}a_x=0$, we have
\begin{align*}
& f=\sum_{x} a_x \tau_x h=\sum_{x=m_{\mathbf{a}}}^{M_{\mathbf{a}}}a_x \tau_x h \\
&=a_{M_{\mathbf{a}}} (\tau_{M_{\mathbf{a}}} h - \tau_{M_{\mathbf{a}}-1} h)+ (a_{M_{\mathbf{a}}} + a_{M_{\mathbf{a}}-1}) (\tau_{M_{\mathbf{a}}-1} h - \tau_{M_{\mathbf{a}}-2} h) + \dots \\
& +(a_{M_{\mathbf{a}}} + a_{M_{\mathbf{a}}-1}\dots +  a_{m_{\mathbf{a}}})  (\tau_{m_{\mathbf{a}}} h - \tau_{m_{\mathbf{a}}-1} h)  \\
& = \sum_{x=m_{\mathbf{a}}}^{M_{\mathbf{a}}} (\sum_{y=x}^{M_{\mathbf{a}}}  a_y) (\tau_x h -\tau_{x-1}h) = \tau ( \sum_{x=m_{\mathbf{a}}}^{M_{\mathbf{a}}} (\sum_{y=x}^{M_{\mathbf{a}}}  a_y) \tau_{x-1} h) - \sum_{x=m_{\mathbf{a}}}^{M_{\mathbf{a}}} (\sum_{y=x}^{M_{\mathbf{a}}}  a_y) \tau_{x-1} h.
\end{align*}
Therefore, $f=\tau g -g $ with $g=\sum_{x=m_{\mathbf{a}}}^{M_{\mathbf{a}}} (\sum_{y=x}^{M_{\mathbf{a}}}  a_y) \tau_{x-1} h \in \mathcal{D}_0$.

Moreover, if $\{ \tau_x h\}_{x \in \Z}$ are orthogonal in $L^2(\mu)$, 
\begin{align*}
\langle g^2\rangle  & = \langle (\sum_{x=m_{\mathbf{a}}}^{M_{\mathbf{a}}} (\sum_{y=x}^{M_{\mathbf{a}}}  a_y) \tau_{x-1} h)^2 \rangle   = \sum_{x=m_{\mathbf{a}}}^{M_{\mathbf{a}}} (\sum_{y=x}^{M_{\mathbf{a}}}  a_y)^2 \langle h^2\rangle   \\
& \le \sum_{x=m_{\mathbf{a}}}^{M_{\mathbf{a}}} s_{\mathbf{a}} \sum_{y=x}^{M_{\mathbf{a}}}  a_y^2 \langle h^2\rangle   \le s_{\mathbf{a}}^2 \sum_{x=m_{\mathbf{a}}}^{M_{\mathbf{a}}} a_x^2 \langle h^2\rangle  
\end{align*}

\end{proof}

Now, we consider a generalized Fourier series in the space $L^2(\mu)$.  
Let $\N_0:=\{0,1,2,\dots\}$ and $\{ \phi_{n} \}_{n \in \N_0}$ be a countable orthonormal basis of $L^2(\nu)$ satisfying $\phi_0 \equiv 1$. The existence of the countable orthonormal basis follows from the assumption that $L^2(\nu)$ is separable. Let us introduce the multi-index space $\Theta:=\{ \mathbf{n}=(n_x)_{x \in \Z} \in \N_0^{\Z};  |\{x \in \Z; n_x \neq 0\}| < \infty \}$. Then, the set of functions $\{ \phi_{\mathbf{n}}\}_{\mathbf{n} \in \Theta}$ is the countable orthonormal basis of $L^2(\mu)$ where $\phi_{\mathbf{n}}(\eta):=\Pi_{x \in \Z} \phi_{n_x}(\eta_x)$. In particular, if $f \in L^2(\mu)$, then $f=\sum_{\mathbf{n} \in \Theta } \tilde{f}_{\mathbf{n}}\phi_{\mathbf{n}}$ with $\tilde{f}_{\mathbf{n}}=\langle  f \phi_{\mathbf{n}} \rangle $. 

We define the shift operator $(\tau_z \mathbf{n})_x = n_{x-z}$ and $\Theta_*:=\{ \mathbf{n}=(n_x)_{x \in \Z} \in \Theta;  n_x=0 \ (\forall x <0) \ , n_0 \neq 0 \}$. Then, for any $\mathbf{n} \in \Theta \setminus \{\mathbf{0}\}$, there exists a unique pair $(x, \mathbf{n}_*) \in \Z \times \Theta_*$ such that $\mathbf{n}=\tau_x \mathbf{n}_*$.

The next lemma is about the locality of the Fourier series.

\begin{lemma}\label{lem:local}
For amy $f \in \mathcal{D}_0$ and $\mathbf{n}_* \in \Theta_*$, $\tilde{f}_{\tau_x \mathbf{n}_*}=0$ if $|x| \ge s_f+1$. In particular, $(\tilde{f}_{\tau_x \mathbf{n}_*})_{x \in \Z} \in \ell^2_c$.

Moreover, for $\mathbf{n}_* \in \Theta_*$ satisfying ${n_*}_y \neq 0$ with some $|y| \ge 2s_f+1$, $\tilde{f}_{\tau_x \mathbf{n}_*}=0$ for all $x \in \Z$. 
\end{lemma}
\begin{proof}
For $|x| \ge s_f+1$, $\phi_{{n_*}_0}(\eta_x)$ and $f$ are independent and $\langle \phi_{{n_*}_0}(\eta_x)\rangle =0$, so $\langle \phi_{{n_*}_0}(\eta_x) \Pi_{y \in \Z \setminus \{0\}} \phi_{{n_*}_y}(\eta_{x+y}) f\rangle =0.$

Similarly, if $\mathbf{n}_* \in \Theta_*$ satisfying ${n_*}_y \neq 0$ with some $|y| \ge 2s_f+1$, then $\phi_{{n_*}_0}(\eta_x)$ and $f$ are independent for $|x| \le -s_f-1$ and $\phi_{{n_*}_y}(\eta_{x+y})$ and $f$ are independent for $|x| \ge -s_f$ so we have $\tilde{f}_{\tau_x \mathbf{n}_*}=0$ for both cases.
\end{proof}

The next lemma is simple but one of the keys of our main result.

\begin{lemma}\label{lem:norm}
For $f \in \mathcal{D}_0, \|f\|^2=\sum_{\mathbf{n}_* \in \Theta_*} (\sum_{x \in \Z} \tilde{f}_{\tau_x \mathbf{n}_*})^2 $.
\end{lemma}
\begin{proof}
Since $\langle f\rangle =0$, $f_{\mathbf{0}}=0$. Then, by the general observation, $f=\sum_{\mathbf{n} \in \Theta \setminus \{ \mathbf{0}\} } \tilde{f}_{\mathbf{n}}\phi_{\mathbf{n}}=\sum_{\mathbf{n}_* \in \Theta_* } \sum_{x \in \Z} \tilde{f}_{\tau_x\mathbf{n}_*}\phi_{\tau_x\mathbf{n}_*}$. Then,
\begin{displaymath}
\|f\|^2=\sum_{z \in \Z} \langle f \tau_z f\rangle = \sum_{z \in \Z} \langle  \big(\sum_{\mathbf{n}_* \in \Theta_* } \sum_{x \in \Z} \tilde{f}_{\tau_x\mathbf{n}_*}\phi_{\tau_x\mathbf{n}_*} \big) \big( \sum_{\mathbf{n}_*' \in \Theta_* } \sum_{x' \in \Z} \tilde{f}_{\tau_{x'}\mathbf{n}_*'}\phi_{\tau_{x'+z}\mathbf{n}_*'} \big)\rangle .
\end{displaymath}
Since $\sum_{\mathbf{n}_* \in \Theta_* } \sum_{x \in \Z} \tilde{f}_{\tau_x\mathbf{n}_*}^2 < \infty$ and $\{ \phi_{\mathbf{n}}\}_{\mathbf{n} \in \Theta}$ is an orthonormal basis, we have
\begin{displaymath}
\langle  \big(\sum_{\mathbf{n}_* \in \Theta_* } \sum_{x \in \Z} \tilde{f}_{\tau_x\mathbf{n}_*}\phi_{\tau_x\mathbf{n}_*} \big) \big( \sum_{\mathbf{n}_*' \in \Theta_* } \sum_{x' \in \Z} \tilde{f}_{\tau_{x'}\mathbf{n}_*'}\phi_{\tau_{x'+z}\mathbf{n}_*'} \big)\rangle  = \sum_{\mathbf{n}_* \in \Theta_* } \sum_{x \in \Z} \tilde{f}_{\tau_x\mathbf{n}_*}\tilde{f}_{\tau_{x-z} \mathbf{n}_*}.
\end{displaymath}
Therefore, 
\[
\|f\|^2=\sum_{z \in \Z}\sum_{\mathbf{n}_* \in \Theta_* } \sum_{x \in \Z} \tilde{f}_{\tau_x\mathbf{n}_*}\tilde{f}_{\tau_{x-z} \mathbf{n}_*}= \sum_{\mathbf{n}_* \in \Theta_* }\big(\sum_{x \in \Z} \tilde{f}_{\tau_x\mathbf{n}_*} \big)^2.
\]
\end{proof}

\begin{proposition}
If $f \in C_0$, then $f \in \mathfrak{G}$.
\end{proposition}
\begin{proof}
By Lemma \ref{lem:norm}, $\|f\|=0$ implies $\sum_{x \in \Z} \tilde{f}_{\tau_x\mathbf{n}_*}=0$ for any $\mathbf{n}_* \in \Theta_*$. Then, combining the fact that $\phi_{\mathbf{n}_*} \in \mathcal{D}_0$ for each $\mathbf{n}_* \in \Theta_*$ with Lemma \ref{lem:grad}, for each fixed $\mathbf{n}_* \in \Theta_*$, there exists $ g_{\mathbf{n}_*}  \in \mathcal{D}_0$ such that $\sum_{x \in \Z} \tilde{f}_{\tau_x\mathbf{n}_*}\tau_x \phi_{\mathbf{n}_*}=\tau g_{\mathbf{n}_*}-g_{\mathbf{n}_*}$. Moreover, since $\{ \tau_x \phi_{\mathbf{n}_*}\}_{x \in \Z}$ are orthogonal and $\tilde{f}_{\tau_x\mathbf{n}_*}=0$ for $|x| \ge s_f+1$ by Lemma \ref{lem:local}, 
\[
\langle g_{\mathbf{n}_*}^2\rangle  \le (2s_f+1)^2 \sum_{x \in \Z} \tilde{f}_{\tau_x\mathbf{n}_*}^2 \langle \phi_{\mathbf{n}_*}^2\rangle =(2s_f+1)^2 \sum_{x \in \Z} \tilde{f}_{\tau_x\mathbf{n}_*}^2.
\]

 By the construction, $\{g_{\mathbf{n}_*}\}_{\mathbf{n}_*} $ are orthogonal in $L^2(\mu)$ and so $g:=\sum_{\mathbf{n}_* \in \Theta_* } g_{\mathbf{n}_*} \in L^2(\mu)$ since
\[
\langle g^2\rangle =\sum_{\mathbf{n}_* \in \Theta_* }\langle g_{\mathbf{n}_*}^2\rangle  \le (2s_f+1)^2 \sum_{\mathbf{n}_* \in \Theta_* } \sum_{x \in \Z} \tilde{f}_{\tau_x\mathbf{n}_*}^2 =(2s_f+1)^2\langle f^2\rangle .
\]
Also, $\langle g\rangle =0$. The locality of $g$ follows from the following two facts: (i) $g_{\mathbf{n}_*}=0$ if $\mathbf{n}_* \in \Theta_*$ satisfying ${n_*}_y \neq 0$ with some $|y| \ge 2s_f+1$ by Lemma \ref{lem:local}, (ii) the support of $g_{\mathbf{n}_*}$ is included in the union of the support of $\{\tau_x \phi_{\mathbf{n}_*}\}_{-s_f \le x \le s_f}$. Therefore, $g \in \mathcal{D}_0$.

Finally, we see that since 
\[
f=\sum_{\mathbf{n}_* \in \Theta_* } \sum_{x \in \Z} \tilde{f}_{\tau_x\mathbf{n}_*}\phi_{\tau_x\mathbf{n}_*}= \sum_{\mathbf{n}_* \in \Theta_* } \sum_{x \in \Z} \tilde{f}_{\tau_x\mathbf{n}_*}\tau_x \phi_{\mathbf{n}_*}= \sum_{\mathbf{n}_* \in \Theta_* }(\tau g_{\mathbf{n}_*} - g_{\mathbf{n}_*}) = \tau g -g 
\]
which implies $f \in \mathfrak{G}$.

\end{proof}


\subsection{Multi-dimensional setting}

In this subsection, we generalize our result to the multi-dimensional setting. 

Let $(X, \mathcal{F},\nu)$ be an probability space where $L^2(\nu)$ is separable and $\Omega:=X^{\Z^d}$ be the infinite product probability space equipped with the probability measure $\mu:=\nu^{\Z^d}$. 
 
Let $\mathcal{D}_0:=\{f \in L^2(\mu); f : \textit{local}, \langle f\rangle =0 \} $. For $f \in \mathcal{D}_0$, we define a semi-norm $\| \cdot \|$ as
\begin{displaymath}
\| f \|^2:=\lim_{k \to \infty}\frac{1}{(2k+1)^d} \langle (\sum_{|x| \le k}\tau_xf)^2\rangle =\sum_{x \in \Z^d} \langle f \tau_x f\rangle = \langle f | f\rangle.
\end{displaymath}

Recall that 
\[
\mathfrak{G}:=\{ \sum_{\a=1}^d (\tau^{\a} g_{\a}- g_{\a}); g_{\a} \in \mathcal{D}, \a=1,2,\dots,d \}=\{ \sum_{\a=1}^d (\tau^{\a} g_{\a}- g_{\a}); g_{\a} \in \mathcal{D}_0, \a=1,2,\dots,d \}
\]
and $C_0:=\{f \in \mathcal{D}_0; \|f \|=0 \}$.

Let $\ell^2_c:=\{\mathbf{a}=(a_x)_{x \in \Z^d} \in \R^{\Z^d}; |\{x \in \Z^d; a_x \neq 0\}| < \infty \}$. For $\mathbf{a} \in \ell^c_2$ satisfying $\mathbf{a} \nequiv \mathbf{0}$, define $M_{\mathbf{a}}:=\max_{1 \le \a \le d} \max\{ x \in \Z; \exists a_y \neq 0 \textit{ s.t. } y_{\a}=x\}$ and $m_{\mathbf{a}}:=\min_{1 \le \a \le d} \min\{ x \in \Z; \exists a_y \neq 0 \textit{ s.t. } y_{\a}=x\}$. As a convention, take $M_{\mathbf{0}}=m_{\mathbf{0}}
=0$. We also define $\Lambda_{\mathbf{a}}:=\{ x \in \Z^d; m_{\mathbf{a}} \le x_{\a} \le M_{\mathbf{a}}, \a=1,2,\dots,d \}$ and $s_{\mathbf{a}}:= |\Lambda_{\mathbf{a}}|$. Here, the only essential property is that the hypercube $\Lambda_{\mathbf{a}}$ satisfies $\Lambda_{\mathbf{a}} \supset \{x \in \Z^d; a_x \neq 0\}$.

\begin{lemma}\label{lem:gradhigh}
Let $f \in \mathcal{D}_0$ and assume that there exists $(a_x)_{x \in \Z^d} \in \ell^2_c$ and $h \in \mathcal{D}_0$ satisfying $f=\sum_{x \in \Z^d} a_x \tau_x h$ and $\sum_{x \in \Z^d} a_x=0$. Then, there exists a $d$-tuple of functions $(g_1,g_2,\dots,g_d) \in (\mathcal{D}_0)^d$ such that $f=\sum_{\a=1}^d (\tau^{\a} g_{\a} -g_{\a}) $, hence $f \in \mathfrak{G}$. Moreover, if $\{ \tau_x h\}_{x \in \Z^d}$ are orthogonal in $L^2(\mu)$, then we can find such a $d$-tuple of functions $(g_1,g_2,\dots,g_d)$ which satisfy also $\sum_{\a=1}^d \langle g_{\a}^2\rangle  \le  s_{\mathbf{a}}^2 \sum_{x \in \Z^d}a_x^2 \langle h^2\rangle $.
\end{lemma}
\begin{proof}
The strategy of the proof is same as the proof of Lemma \ref{lem:grad}. We first construct a one-to-one function $\psi: \{1,2,\dots,s_{\mathbf{a}}\} \to \La_{\mathbf{a}} $ satisfying the property that $|\psi(k)-\psi(k+1)|=1$ for any $1 \le k \le s_{\mathbf{a}}-1$. This is done explicitly and elementary (cf. Appendix 3 Lemma 4.9 in \cite{KL}). Then, repeat the technique used in Lemma \ref{lem:grad} as follows:
\begin{align*}
& f=\sum_{x} a_x \tau_x h=\sum_{x \in \La_{\mathbf{a}} }a_x \tau_x h  = \sum_{k=1}^{s_{\mathbf{a}}}a_{\psi(k)}\tau_{\psi(k)} h\\
&=a_{\psi(s_{\mathbf{a}})} (\tau_{\psi(s_{\mathbf{a}})} h - \tau_{\psi(s_{\mathbf{a}}-1)} h)+ (a_{\psi(s_{\mathbf{a}})} + a_{\psi(s_{\mathbf{a}}-1)}) (\tau_{\psi(s_{\mathbf{a}}-1)} h - \tau_{\psi(s_{\mathbf{a}}-2)} h) + \dots \\
& +(a_{\psi(s_{\mathbf{a}})} + a_{\psi(s_{\mathbf{a}}-1)}\dots +  a_{\psi(1)})  \tau_{\psi(1)} h   \\
& = \sum_{k=2}^{s_{\mathbf{a}}} (\sum_{l=k}^{s_{\mathbf{a}}}  a_{\psi(l)}) (\tau_{\psi(k)} h -\tau_{\psi(k-1)}h) \\
& = \sum_{\a=1}^d \sum_{k=2}^{s_{\mathbf{a}}} (\sum_{l=k}^{s_{\mathbf{a}}}  a_{\psi(l)}) (\tau_{\psi(k)} h -\tau_{\psi(k-1)}h)(\mathbf{1}_{\{\psi(k)-\psi(k-1)=e_{\a} \}} + \mathbf{1}_{\{\psi(k)-\psi(k-1)=- e_{\a} \}}) \\
& = \sum_{\a=1}^d \sum_{k=2}^{s_{\mathbf{a}}} (\sum_{l=k}^{s_{\mathbf{a}}}  a_{\psi(l)}) (\tau_{\psi(k)} h -\tau_{\psi(k-1)}h)(\mathbf{1}_{\{\psi(k)= \tau^{\a} \psi(k-1) \}} + \mathbf{1}_{\{\psi(k-1)= \tau^{\a}\psi(k) \}}) \\
& = \sum_{\a=1}^d (\tau^{\a}g_{\a}-g_{\a})
\end{align*}
where 
\[
g_{\a}=\sum_{k=2}^{s_{\mathbf{a}}} (\sum_{l=k}^{s_{\mathbf{a}}}  a_{\psi(l)}) \tau_{\psi(k-1)} h \mathbf{1}_{\{\psi(k)= \tau^{\a} \psi(k-1) \}} - \sum_{k=2}^{s_{\mathbf{a}}} (\sum_{l=k}^{s_{\mathbf{a}}}  a_{\psi(l)}) \tau_{\psi(k)} h \mathbf{1}_{\{\psi(k-1)= \tau^{\a} \psi(k) \}}.
\]
Note that, for any $k$, $\psi(k)= \tau^{\a} \psi(k-1)$ implies $\psi(k-2) \neq \tau^{\a} \psi(k-1)$ since $\psi$ is one-to-one. Namely, if $\{ \tau_x h \}_{x \in \Z^d}$ are orthogonal in $L^2(\mu)$, then the two terms in the definition of $g_{\a}$ are orthogonal and
\begin{align*}
\langle g_{\a}^2 \rangle & = \sum_{k=2}^{s_{\mathbf{a}}}  \langle (\sum_{l=k}^{s_{\mathbf{a}}}  a_{\psi(l)})^2 (\tau_{\psi(k-1)} h)^2 \mathbf{1}_{\{\psi(k)= \tau^{\a} \psi(k-1) \}} \rangle \\
& + \sum_{k=2}^{s_{\mathbf{a}}} \langle  (\sum_{l=k}^{s_{\mathbf{a}}}  a_{\psi(l)})^2 (\tau_{\psi(k)} h)^2 \mathbf{1}_{\{\psi(k-1)= \tau^{\a} \psi(k) \}} \rangle \\
& =  \sum_{k=2}^{s_{\mathbf{a}}} (\mathbf{1}_{\{\psi(k)= \tau^{\a} \psi(k-1) \}} + \mathbf{1}_{\{\psi(k-1)= \tau^{\a} \psi(k) \}}) (\sum_{l=k}^{s_{\mathbf{a}}}  a_{\psi(l)})^2  \langle h^2  \rangle \\
& \le \sum_{k=2}^{s_{\mathbf{a}}}(\mathbf{1}_{\{\psi(k)= \tau^{\a} \psi(k-1) \}} + \mathbf{1}_{\{\psi(k-1)= \tau^{\a} \psi(k) \}})  s_{\mathbf{a}} \sum_{l=k}^{s_{\mathbf{a}}}  (a_{\psi(l)})^2  \langle h^2  \rangle \\
& \le   s_{\mathbf{a}} \sum_{x \in \Z^d} a_x^2 \langle h^2  \rangle \sum_{k=2}^{s_{\mathbf{a}}}(\mathbf{1}_{\{\psi(k)= \tau^{\a} \psi(k-1) \}} + \mathbf{1}_{\{\psi(k-1)= \tau^{\a} \psi(k) \}}).  
\end{align*}
Then, since $\sum_{\a=1}^d (\mathbf{1}_{\{\psi(k)= \tau^{\a} \psi(k-1) \}} + \mathbf{1}_{\{\psi(k-1)= \tau^{\a} \psi(k) \}}) =1$ for any $k$, we complete the proof.
\end{proof} 
 
\begin{remark}
The uniqueness result in Lemma \ref{lem:grad} fails for the multi-dimensional case. Note that we do not use the uniqueness result anywhere in the paper.
\end{remark}

For the part of the generalize Fourier series, we do not need to change the strategy. Note that we define $\Theta_*$ as the quotient of $\Theta \setminus \{\mathbf{0}\}$ by the equivalence relation $\mathbf{n} \sim \mathbf{n}'$ if any only if there exists $x \in \Z^d$ such that $\tau_x \mathbf{n}=\mathbf{n}'$.

To make clear the locality of the Fourier series, we introduce the following notation.

For $\mathbf{n}_* \in \Theta_*$, let $\text{rad}(\mathbf{n}_*)= \max_{1 \le \a \le d} \max\{ |x_{\a}-x'_{\a}|; {n_*}_x \neq 0, {n_*}_{x'} \neq 0 \}$. 

\begin{lemma}\label{lem:local}
For $\mathbf{n}_* \in \Theta_*$ satisfying $\text{rad}(\mathbf{n}_*) \ge 2s_f+1$, $\tilde{f}_{\tau_x \mathbf{n}_*}=0$ for all $x \in \Z^d$. 

Moreover, if $\text{rad}(\mathbf{n}_*) \le 2s_f$, then we can choose the representative $\mathbf{n}_*$ so as $\{ x \in \Z^d; {n_*}_x \neq 0 \} \subset \{x \in \Z^d; -s_f \le x_{\a} \le  s_f, \a=1,2,\dots,d \} $. Then, for this representative, $\tilde{f}_{\tau_x \mathbf{n}_*}=0$ if $|x| \ge s_f+1$.
\end{lemma}

The next lemma holds in the same way as the one-dimensional case.

\begin{lemma}\label{lem:norm}
For $f \in \mathcal{D}_0, \|f\|=\sum_{\mathbf{n}_* \in \Theta_*} (\sum_{x \in \Z^d} \tilde{f}_{\tau_x \mathbf{n}_*})^2 $.
\end{lemma}

Our main result also follows in the same way. Just note that $\sum_{x \in \Z^d}\tilde{f}_{\tau_x \mathbf{n}_*}\tau_x \phi_{\mathbf{n}_*}$ does not depend on the choice of the representative of $\mathbf{n}_*$.

\begin{proposition}
If $f \in C_0$, then $f \in \mathfrak{G}$.
\end{proposition}

Hence, we prove $C_0 \subset \mathfrak{G}$ and so Theorem \ref{thm:main}.

\section{Application to the stochastic energy transport model}\label{sec:energy}

In this section, we show an application of our result to one specific model called stochastic energy transport model, which is paid much attention from particularly physical point of view. See more detailed background of the model in \cite{GG08,GG09}.

The model is heuristically obtained as a mesoscopic energy transport model from a microscopic mechanical dynamics consist of a one-dimensional array of two-dimensional cells, each containing a single hard-disc particle or an array of three-dimensional cells, each containing a single hard-sphere particle. 

This mesoscopic model completely fits to our general setting taking $(X,\F,\mu)=((0,\infty),\mathcal{B}((0,\infty),\nu)$ where
\begin{equation}\label{eq:equilibriummeasure}
\nu(d\eta)=\frac{\eta^{\frac{d}{2}-1}\exp(-\frac{\eta}{T})}{T^{\frac{d}{2}}\Gamma(\frac{d}{2})}d\eta
\end{equation}
with a given model parameter $d$ and the temperature $T$.

The operator $L$ is the generator of the infinite volume dynamics, given as $Lf=\sum_{x \in \Z}(L_{x,x+1}+L_{x+1,x})f$ where
\[
L_{x,x+1}f=\frac{1}{2}\int_{-\eta_{x+1}}^{\eta_x} du[W(\eta_x-\eta,\eta_{x+1}+u| \eta_x,\eta_{x+1})f(\dots,\eta_x-u,\eta_{x+1}+u, \dots) -f(\eta)]
\]
and $L_{x,x+1}=L_{x+1,x}$
where $W(\eta_a,\eta_b| \eta_a-u,\eta_b+u)$ describes the rate of exchange of energy $u$ between sites $a$ and $b$ at respective energies $\eta_a$ and $\eta_b$. In other words, in this dynamics, the amount of energy $u$ is moved between the neighboring sites $a$ and $b$ with rate $W(\eta_a,\eta_b| \eta_a-u,\eta_b+u)$. The specific forms of the kernel should be found in \cite{GG08,GG09}. The dynamics obviously conserves the sum of the energies, hence $\xi(\eta)=\eta$.

Under the diffusive space-time scaling limit, the time evolution of the local temperature will be given by 
\begin{equation}\label{eq:Fourier}
\partial_t T=\partial_x (D(T)\partial_x T), \quad  T=T(x,t)
\end{equation}
with thermal diffusivity $D(T)$. In \cite{GG08,GG09}, the authors conjectured that $D(T)=D^s(T)$ where $D^s(T)$ is the static part of the thermal diffusivity. However, with our main result, $D(T)=D^s(T)$ implies the energy current is the gradient and it is not true, hence we conclude that the conjecture fails.
Recently, Gilbert and Gaspard \cite{GG16}
show how the variational characterization can be put to use to obtain the correction to static (or instantaneous) part of the diffusion coefficient and carried out further molecular dynamics simulations, which on one side confirm our picture and on the other hand also show that the correction is very small.

\section*{Acknowledgement}

The author expresses her sincere thanks to Professor Herbert Spohn and Professor Stefano Olla for their insightful discussions and encouragement.

\end{document}